\documentclass[11pt]{article}
\usepackage[a4paper, total={6in, 8in}]{geometry}
\usepackage{enumerate}
\usepackage{amsmath}
\usepackage{amssymb,latexsym}
\usepackage{amsthm}
\usepackage{color}
\usepackage{cancel}
\usepackage{graphicx}

\usepackage{hyperref}
\usepackage{comment}
\usepackage{amscd}
\usepackage{cite}
\makeatletter
 \usepackage{comment}
\let\wfs@comment@comment\comment
\let\comment\@undefined

\usepackage{changes}
\let\wfs@changes@comment\comment
\let\comment\@undefined

\newcommand\comment{%
    \ifthenelse{\equal{\@currenvir}{comment}}
    {\wfs@comment@comment}
    {\wfs@changes@comment}%
}

\definechangesauthor[name=Daniele, color=red]{DAN}
\definechangesauthor[name=Nico, color=blue]{NICO}
\newtheorem{theorem}{Theorem}[section]

\newtheorem{lemma}[theorem]{Lemma}

\newtheorem{definition}[theorem]{Definition}
\newtheorem{proposition}[theorem]{Proposition}

\newtheorem*{theorem*}{Main Theorem}

\newcommand\PG{\mathrm{PG}}
\newcommand\Q{\mathrm{Q}}
\newcommand\mO{\mathrm{O}}
\newcommand\fq{\mathbb{F}_q}
\newcommand\fp{\mathbb{F}_p}
\newcommand\ft{\mathbb{F}_3}
\newcommand\fqn{\mathbb{F}_{q^n}}

\title{On the classification of low-degree ovoids of $Q(4,q)$}
\author{Daniele Bartoli\thanks{Department of Mathematics and Informatics, University of Perugia, Perugia, Italy.
Email address: daniele.bartoli@unipg.it}\hspace{0.15 cm}  and Nicola Durante\thanks{Department of Mathematics and applications R. Caccioppoli, University of Naples Federico II, Naples, Italy.
Email address: ndurante@unina.it} }
\date{}
\begin{document}

\maketitle

\begin{abstract}
Ovoids of the non-degenerate quadric $Q(4,q)$ of $\PG(4,q)$ have been studied since the end of {the} '80s.
They are rare objects and, beside the classical example given by an elliptic quadric,  only three classes are known for $q$ odd, one class for $q$ even, and a sporadic example for $q=3^5.$ It is well known that to any ovoid of $Q(4,q)$ a {bivariate} polynomial $f(x,y)$ can be associated.
In this paper we classify ovoids of $Q(4,q)$ whose corresponding polynomial $f(x,y)$ has ``low degree" compared with $q$, in particular   $deg(f)<(\frac{1}{6.3}q)^{\frac{3}{13}}-1.$ 
Finally, as an application, {two classes} of permutation polynomials in characteristic $3$ are obtained.    
\end{abstract}
\section{Introduction}\label{Sec:Intro}

Let  $\fqn$, $q=p^h$, be the finite field with $q^n$ elements. Denote by  $tr(x)=x+x^p+\cdots x^{p^{hn-1}}$ the {\em absolute trace} and by $N(x)=N_{\fqn/\fq}(x)=x^{\frac{q^n-1}{q-1}}$ the {\em norm} of an element $x$ of $\fqn$. For every $i\in \fq$, 
let $N_i=\{x \in \fqn|N(x) = i\}.$

 Denote by $\mathbb{P}$ a {\em finite classical polar space}, i.e. the set of absolute points of either a polarity or a non-singular quadratic form of a projective space $\PG(m,q)$.
The maximal dimensional projective subspaces contained in $\mathbb{P}$ are the {\em generators} of $\mathbb{P}$.

\begin{definition}\label{Definition}
Let $\mathbb{P}$ be a finite classical polar space of $\PG(m,q)$. An {\em ovoid} of $\mathbb{P}$ is a set of points of $\mathbb{P}$ that has exactly one point in common with each generator of $\mathbb{P}.$
An ovoid ${\mO}$ of $\mathbb{P}$ is a {\em translation} ovoid with respect to a point $P {\in}{ \mO}$ if there is a collineation group of $\mathbb{P}$  fixing $P$ linewise (i.e. stabilizing all lines through $P$) and acting sharply transitively on the points of ${\mO} \setminus \{P\}$.
\end{definition}

 Regarding ovoids of the parabolic quadric $Q(m,q)$ of $\PG(m,q)$, $m$ even, it is known that there are no ovoids if $m>6$ (see \cite{Tha81,GunMoo}).
Hence for parabolic quadrics the only open problems concern $m=4$ and $m=6$. 
For $m=6$, there are no ovoids if $q$ is a prime ($q\ne 3$) and there are two families of ovoids for $Q(6,3^h)$; see \cite{Tha81,OkeTha,BalGovSto,CamThaPay,Tit,Tha80}.

\noindent For more results on ovoids of  finite classical polar spaces see e.g. \cite{DebKleMet}.


\noindent In this paper we will focus on ovoids of $Q(4,q).$  {Denote by $(X_0,X_1,X_2,X_3,X_4)$ the homogenous coordinates of $\PG(4,q)$ and by $H_\infty$  the hyperplane at infinity  $X_0=0$.} 

\noindent Up to collineation, in $\PG(4,q)$ there is a unique non-degenerate quadric, the parabolic quadric $Q(4,q)$. In the sequel we fix the following parabolic quadric:

\[ \Q:  X_0X_4+X_1X_3+X_2^2=0.\]

Two points of $Q(4,q)$ are non-collinear if they are not contained in any generator of $Q(4,q)$. It can be easily shown that an ovoid of $Q(4,q)$ is a set of $q^2+1$ pairwise non-collinear  points  on the quadric.

We would like also to underline that ovoids of $Q(4,q)$ are related to many  relevant objects in finite projective spaces such as ovoids of  $\PG(3,q)$ (if $q$ is even) see e.g. \cite{Bal2},\cite{Tit}, semifield flocks of a quadratic cone of $\PG(3,q)$, eggs of finite projective spaces, translation generalized quadrangles and $C_F^\sigma$-sets of $\PG(2,q)$ (if $q$ is odd) (see e.g. \cite{Law},\cite{Lun},\cite{Dur19}). Also, ovoids of $Q(4,q)$ are equivalent to symplectic spreads of $\PG(3, q)$ via the isomorphism between $Q(4,q)$ and the dual of the symplectic generalized quadrangle $W(3,q)$ see \cite{Tha72}.

Since the collineation group of $Q(4,q)$  acts transitive on pairs non-collinear points,  we can always assume that an ovoid $\mO_4$ of $Q(4,q)$ contains the points $(1,0,0,0,0)$ and $(0,0,0,0,1)$ and hence it 
can be written in the following form:
\begin{equation}\label{Eq:Param}
\mO_4(f) = \{(1,x,y,f(x,y),-y^2-xf(x,y)\}_{x,y\in \fq} \cup \{(0,0,0,0,1)\}, \end{equation}

\noindent for some function $f:\fq^2 \longrightarrow \fq$ with $f(0,0) =0$.

\noindent  {The set $\mO_4(f)$ is an ovoid if and only if }
\begin{equation}\label{Eq:Intro2} \langle (1,x_1,y_1,f(x_1,y_1),-y_1^2-x_1f(x_1,y_1), (1,x_2,y_2,f(x_2,y_2),-y_2^2-x_2f(x_2,y_2)\rangle \ne 0,    
\end{equation}
 
\noindent for every $(x_1,y_1) \ne (x_2,y_2)$ in $\mathbb{F}_q^2$, where $\langle  \cdot,\cdot \rangle$ is the symmetric for $q$ odd (or alternating for $q$ even) form associated to the quadratic form of the quadric $Q(4,q)$.

Note that  Condition \eqref{Eq:Intro2} reads

\begin{equation}\label{Eq:intro}
(y_1-y_2)^2 + (x_1-x_2)(f(x_1,y_1)-f(x_2,y_2)) \ne 0,\end{equation}
for every   $(x_1,y_1) \ne (x_2,y_2)$ in $\mathbb{F}_q^2$.

The only known classification theorems  for ovoids of $Q(4,q)$ are the following.

\begin{theorem}\cite{BalGovSto}
If ${\mO}$ is an ovoid of $Q(4,q)$, $q$ odd prime,  then ${\mO}$ is an elliptic quadric. 
\end{theorem}

\begin{theorem}\cite{Bal3,Lav06}\label{thm14}
 Let $f(X,Y)=g(X)+h(Y)$, with both $g$ and $h$ $\mathbb{F}_{q^{\prime}}$-linear maps, $q=(q^{\prime})^n$, $q$ odd.
 If $q^{\prime}\geq 4n^2-8n+2$ or $q^{\prime}$ is a prime and 
 $q^{\prime} > 2n^2 -(4-2\sqrt{3})n + (3-2\sqrt{3})$, then  ${\mO}_4(f)$ is either an elliptic quadric or a Kantor ovoid.
\end{theorem}

Also, since ovoids of $Q(4,q)$ with $f$ as in Theorem \ref{thm14} are equivalent to rank $2$ commutative semifields the following also is known.

\begin{theorem}\cite{LavRod}
Let $f(X,Y)=g(X)+h(Y)$, with both $g$ and $h$ $\mathbb{F}_{q^{\prime}}$-linear maps, $q=(q^{\prime})^n$, $q$ odd. If either $n\le 4$ or $(q^{\prime},n)=(3,5)$, then $O_4(f)$ is either an elliptic quadric or a Kantor ovoid or a Thas-Payne ovoid or a Penttila-Williams ovoid.
\end{theorem}

{In Section \ref{Sec:AlgebraicVarieties}, using \eqref{Eq:intro}, a hypersurface $\mathcal{S}_f$ of $\PG(4,q)$ will be attached to $\mO_4(f)$. A similar connection between ovoids in $\PG(2,q)$ (hyperovals), $q$ even, and algebraic curves over finite fields  was used in \cite{CS2015,HM2012} to classify hyperovals corresponding to $o$-polynomials of small degree. Investigating the existence of absolutely irreducible components defined over $\mathbb{F}_q$ in  $\mathcal{S}_f$, we will provide the following classification result, which is also the main achievement of our paper.
\begin{theorem*}
Let $f(X,Y)=\sum_{ij}a_{i,j}X^iY^j$, with $a_{0,1}=0$ if $p>2$. Suppose that $q>6.3 (d+1)^{13/3}$, where $\deg(f)=d$. If $\mO_4(f)$ is an ovoid of $\Q$ then one of the following holds:
\begin{enumerate}
    \item $p=2$ and $f(X,Y)=a_{1,0}X+a_{0,1}Y$.  Hence $O_4(f)$ is an elliptic quadric.
    \item $p>2$ and $f(X,Y)=a_{p^j,0}X^{p^j}$, $j\geq 0$.  Hence $O_4(f)$ is either an elliptic quadric or a Kantor ovoid.
\end{enumerate}
\end{theorem*}}

The only known ovoids of $Q(4,q)$ are listed in  Table \ref{Table1}, where  $n$ is a non-square of $\fq$, $q=p^h$ odd,   $\sigma\ne 1 $ is an automorphism of $\fq$, and
$a\ne 1$ is an element with $tr(a)=1$ of $\fq$, $q$ even. All but  the last two examples of ovoids in  {Table \ref{Table1}} are translation ovoids of $Q(4,q)$ with respect to their point at infinity.

\begin{center}
\begin{table}[ht]
\caption{Known ovoids of $Q(4,q)$.}\label{Table1}
\vspace{0.25cm}
\begin{center}
 \begin{tabular}{||c c c c||} 
\hline
Name  &  $f(x,y)$ &  Restrictions &Ref. \\ [0.5ex] 
 \hline\hline
{\em Elliptic quadric} & $-nx$ & $q$ odd & \\ 
 \hline
{\em Elliptic quadric} & $ax+y$ & $q$ even  &\\ 
 \hline
{\em Kantor } & $-nx^\sigma$ &  $q$ odd, $h>1$ & \cite{Kantor} \\
 \hline
{\em Penttila-Williams}  & $-x^9-y^{81}$ & $p=3$, $h=5$  & \cite{PenWil}\\
 \hline
{\em Thas-Payne}  & $-nx-(n^{-1} x)^{1/9} - y^{1/3}$ & $p=3, h>2$  & \cite{ThaPay} \\ 
\hline
{\em Ree-Tits slice}  & $-x^{2\sigma+3}-y^\sigma$ & $p=3, h>1$, $h$ odd, $\sigma=\sqrt{3q} $& \cite{Kantor}\\
 \hline
{\em Tits} & $x^{\sigma+1}+ y^\sigma$ & $p=2, h>1$, $h$ odd, $\sigma = \sqrt{2q}$ & \cite{Tit}\\ 
\hline
\end{tabular}
\end{center}
\end{table}
\end{center}

Note that the Ree-Tits slice is obtained from slicing the Ree-Tits ovoid of $Q(6, q)$; see  \cite{Kantor}.

In the sequel we briefly discuss the connection between these examples and   Theorem \ref{Th:Main2}.

In the even characteristic case,  Main Theorem  does not apply to the Tits ovoid
$$f_2(x,y)=x^{2^{h+1}+1}+ y^{2^{h+1}}, \qquad q=2^{2h+1},$$
since $\deg(f_2)\simeq  \sqrt{q}$.

In the odd characteristic case, more examples of ovoids are known. Apart from the case $f_3(x,y)=-nx^{p^r}$, $q=p^h$, (Elliptic quadric and Kantor)  which satisfies Condition 2 in Main Theorem, the two infinite families (Ree-Tits slice and Thas-Payne in characteristic $3$)  arise from $f_4(x,y)=-x^{3^{2\cdot 3^{h+1}+3}}-y^{3^{h+1}}$, when $q=3^{2h+1}$, and $f_5(x,y)=-nx-(n^{-1}x)^{1/9}-y^{1/3}$, $q=3^h$, $h>2$. Note that in these cases, Theorem \ref{Th:Main2}  does not apply to such examples, since $\deg(f_4)\simeq  \sqrt{q}$ and $\deg(f_5)>\sqrt[3]{q}$.

Finally, observe that for the Penttila-Williams sporadic example in $\mathbb{F}_{3^5}$ arising from $f_6(x,y)=-x^9-y^{81}$ the hypothesis of Main Theorem does not hold  since $\deg(f_6)$ is too high with respect to $q=3^5$.

 {This paper is organized as follows.
 The main result is proved in Section \ref{Sec:AlgebraicVarieties}. Finally, in  Section \ref{Sec:permPol} two families of permutation polynomials connected with  Pentilla-William and Thas-Payne ovoids are obtained.}

\section{Proof of the main theorem}\label{Sec:AlgebraicVarieties}
The aim of this section is to provide a partial classification of the ovoids of $Q(4,q)$ via the investigation of specific algebraic varieties; see \cite{BartoliSurvey} for a survey on links between algebraic varieties over finite fields and relevant combinatorial objects. 

An algebraic hypersurface $\mathcal{S}$ is an algebraic variety that may be defined by a single implicit equation. An algebraic hypersurface defined over a field $\mathbb{K}$  is \emph{absolutely irreducible}  if the associated polynomial is irreducible over every algebraic extension of $\mathbb{K}$. An absolutely irreducible $\mathbb{K}$-rational component of a hypersurface $\mathcal{S}$, defined by the polynomial $F$, is simply an absolutely irreducible hypersurface such that the associated polynomial has coefficients in $\mathbb{K}$ and it is a factor of $F$.   As in the previous sections, we will make use of the homogeneous equations for algebraic varieties. For a deeper introduction to algebraic varieties we refer the interested reader to \cite{Hartshorne}.

As mentioned in Section \ref{Sec:Intro} an ovoid of $\Q:  X_0X_4+X_1X_3+X_2^2=0$ can be written in the  form $\mO_4(f) = \{(1,x,y,f(x,y),-y^2-xf(x,y)\}_{x,y\in \fq} \cup \{(0,0,0,0,1)\},$
for some function $f$ with $f(0,0) =0$ and $(y_1-y_2)^2 + (x_1-x_2)(f(x_1,y_1)-f(x_2,y_2)) \ne 0$ for every   $(x_1,y_1) \ne (x_2,y_2)$ in $\mathbb{F}_q^2$. Let $\widetilde f(X,Y,T)$ be the homogenization of $f(X,Y)$ and let $d=\deg(f)$. In order to obtain non-existence results and partial classifications for the ovoids of $\Q$, we will consider the hypersurface  $\mathcal{S}_f\subset \mathrm{PG}(4,q)$ defined by $F(X_0,X_1,X_2,X_3,X_4)=0$, where 
\begin{equation}\label{Eq:F}
F(X_0,X_1,X_2,X_3,X_4)= (X_2 -X_4)^2X_0^{d-1} + (X_1-X_3) \Big(\widetilde f(X_1,X_2,X_0)-\widetilde f(X_3,X_4,X_0)\Big).
\end{equation}

Note that $\mO_4(f)$ is an ovoid of $\Q$ if and only if $\mathcal{S}_f$ contains no affine $\mathbb{F}_q$-rational points off the plane $X_1-X_3=0=X_2-X_4$. 

A fundamental tool to determine the existence of rational points in an algebraic variety over finite fields is the following result by Lang and Weil \cite{LangWeil} in 1954, which can be seen as a generalization of the Hasse-Weil bound.
\begin{theorem}\label{Th:LangWeil}[Lang-Weil Theorem]
Let $\mathcal{V}\subset \mathbb{P}^N(\mathbb{F}_q)$ be an absolutely irreducible variety of dimension $n$ and degree $d$. Then there exists a constant $C$ depending only on $N$, $n$, and $d$ such that 
\begin{equation}\label{Eq:LW}
\left|\#\mathcal{V}(\mathbb{F}_q)-\sum_{i=0}^{n} q^i\right|\leq (d-1)(d-2)q^{n-1/2}+Cq^{n-1}.
\end{equation}
\end{theorem}

\noindent Although the constant $C$ was not  computed in \cite{LangWeil}, explicit estimates have been provided for instance in  \cite{CafureMatera,Ghorpade_Lachaud,Ghorpade_Lachaud2,LN1983,WSchmidt,Bombieri} and they have the general shape $C=f(d)$ provided that $q>g(n,d)$, where $f$ and $g$ are polynomials of (usually) small degree. We refer to \cite{CafureMatera} for a survey on these bounds. Excellent surveys on Hasse-Weil and Lang-Weil type theorems are \cite{Geer,Ghorpade_Lachaud2}. We include here the following result by Cafure and Matera \cite{CafureMatera}. 

\begin{theorem}\cite[Theorem 7.1]{CafureMatera}\label{Th:CafureMatera}
Let $\mathcal{V}\subset\mathrm{AG}(n,q)$ be an absolutely irreducible $\mathbb{F}_q$-variety of dimension $r>0$ and degree $\delta$. If $q>2(r+1)\delta^2$, then the following estimate holds:
$$|\#(\mathcal{V}\cap \mathrm{AG}(n,q))-q^r|\leq (\delta-1)(\delta-2)q^{r-1/2}+5\delta^{13/3} q^{r-1}.$$
\end{theorem}

For our purposes, the existence of an absolutely irreducible $\mathbb{F}_q$-rational component in $\mathcal {S}_f$ is enough to provide asymptotic non-existence results.

\begin{theorem}\label{Th:Main}
Let $\mathcal{S}_f: F(X_0,X_1,X_2,X_3,X_4)=0$, where $F$ is defined as in \eqref{Eq:F}. Suppose that $q>6.3 (d+1)^{13/3}$, $\deg(f)=d$, and $\mathcal{S}_f$ contains an absolutely irreducible component $ \mathcal{V}$ defined over $\mathbb{F}_q$. Then  $\mO_4(f)$ is not an ovoid of $\Q$.
\end{theorem}
\proof
 Since $q>6.3 (d+1)^{13/3}$, by Theorem \ref{Th:CafureMatera}, with $\delta=d+1$ and $r=3$, it is readily seen that $\mathcal{V}$ contains more than $q^2$ affine $\mathbb{F}_q$-rational points. This yields the existence of at least one affine $\mathbb{F}_q$-rational point $(a,b,c,d)$ in $\mathcal{V}\subset \mathcal{S}_f$ such that $a\neq c$ or $b\neq d$ and therefore $\mO_4(f)$ is not an ovoid of $\Q$.
\endproof

We also include, for sake of completeness, the following results  which will be useful  in the sequel. 

\begin{lemma}\cite[Lemma 2.1]{Aubry}\label{Lemma:Aubry} Let $\mathcal{H},\mathcal{S}\subset \mathrm{PG}(n,q)$ be two projective hypersurfaces. If $\mathcal{H} \cap  \mathcal{S}$ has a reduced absolutely irreducible component defined over $\mathbb{F}_q$ then $\mathcal{S}$ has an absolutely irreducible component defined over $\mathbb{F}_q$.
\end{lemma}

\begin{lemma}\cite[Lemma 2.9]{DanYue2020}\label{Lemma:DanYue2020} Let $\mathcal{S}$ be a hypersurface containing $O=(0, 0,\ldots, 0)$ of  affine equation $F(X_1, ..., X_n) =0$, where $F(X_1,\ldots,X_n)=F_d(X_1,\ldots,X_n)+F_{d+1}(X_1,\ldots,X_n)+\cdots$, with $F_i$ the homogeneous part of degree  $i$ of $F(X_1,\ldots,X_n)$ for $i =d, d +1,\dots$. Let $P$ be an $\mathbb{F}_q$-rational simple point of the variety $F_d(X_1,\ldots,X_n)=0$. Then there exists an $\mathbb{F}_q$-rational plane $\pi$ through the line $\ell$ joining $O$ and $P$ such that $\pi \cap \mathcal{S}$ has $\ell$ as a non-repeated tangent $\mathbb{F}_q$-rational line at the origin and $\pi \cap \mathcal{S}$ has a non-repeated absolutely irreducible $\mathbb{F}_q$-rational component.
\end{lemma}


In what follows, $f(X,Y)=\sum_{i,j}a_{i,j}X^iY^j$, with $a_{i,j}\in \mathbb{F}_q$ for all $i,j$, and $d=\deg(f)=\max\{i+j: a_{i,j}\neq0\}$. So, 
\begin{equation}\label{Eq:S_f}
F(X_0,X_1,X_2,X_3,X_4)=(X_2 -X_4)^2X_0^{d-1} + (X_1-X_3) \left[\sum_{i,j} a_{ij}
\left(X_1^i X_2 ^j -X_3^iX_4^j\right)X_0^{d-i-j}\right].
\end{equation}

In view of Theorem \ref{Th:Main}, our goal is to find conditions for which $\mathcal{S}_f$ possesses an absolutely irreducible component defined over $\mathbb{F}_q$.  Recall that $f(0,0)=0$ and therefore   $a_{0,0}=0$. Also, if $p>2$, since $\mathcal{S}_f$ is equivalent to $\mathcal{S}_{f^{\prime}}:F(X_0,X_1,X_2-a_{0,1}X_1/2 ,X_3,X_4-a_{0,1}X_3/2)=0$, where $a_{0,1}^{\prime}=0$, without loss of generality we can assume $a_{0,1}=0$.

\begin{proposition}
Let $\mathcal{S}_f: F(X_0,X_1,X_2,X_3,X_4)=0$, with $F$  as in \eqref{Eq:S_f}. Let $d>1$. If there exists  $i\neq d$ such that $a_{i,d-i}\neq 0$ then $\mathcal{S}_f$ contains an absolutely irreducible component defined over $\mathbb{F}_q$.
\end{proposition}
\begin{proof}
Let $\mathcal{S}^{\prime}=\mathcal{S}_f\cap H_{\infty}$, where $H_{\infty}: X_0=0$. Now, the equation of $\mathcal{S}^{\prime}$ reads 
$$(X_1-X_3)\sum_{i\leq d} a_{i,d-i}
\left(X_1^i X_2 ^{d-i} -X_3^iX_4^{d-i}\right)=0.$$
Note that $(X_1-X_3)\mid G(X_1,X_2,X_3,X_4):=\sum_{i\leq d} a_{i,d-i}
\left(X_1^i X_2 ^{d-i} -X_3^iX_4^{d-i}\right)$ if and only if $G(X_1,X_2,X_1,X_4)=\sum_{i\leq d} a_{i,d-i}X_1 ^{i}
\left(X_2^{d-i} -X_4^{d-i}\right)$ is the zero polynomial. This would be  possible if and only if   $G(X_1,X_2,X_3,X_4)=a_{d,0}\left(X_1^d -X_3^d\right)$, that is $a_{i,d-i}=0$ for all $i\neq d$, a contradiction.

Suppose now that there exists an $i\neq d$ such that $a_{i,d-i}\neq 0$. Since  $(X_1-X_3)$ does not divide  $G(X_1,X_2,X_3,X_4)$, $(X_1-X_3)$ is a non-repeated absolutely irreducible component of $\mathcal{S}^{\prime}$ defined over $\mathbb{F}_q$. By Lemma \ref{Lemma:Aubry}, $\mathcal{S}$ contains an absolutely irreducible component defined over $\mathbb{F}_{q}$ (passing through $H_{\infty}\cap (X_1=X_3))$. 
\end{proof}


In what follows we denote by $\overline{j}$ the nonnegative integer $\max\{j : a_{i,j}\neq 0\}$.
\begin{proposition}\label{Prop:j_bar}
Let $\mathcal{S}_f: F(X_0,X_1,X_2,X_3,X_4)=0$, with $F$  as in \eqref{Eq:S_f}. If $\overline{j}>1$ then $\mathcal{S}_f$ contains an absolutely irreducible component defined over $\mathbb{F}_q$.
\end{proposition}
\begin{proof}

The polynomial $F(X_4,X_1,X_2,X_3,X_0)$ reads 
$$(X_2 -X_0)^2X_4^{d-1} + (X_1-X_3) \left[\sum_{i,j} a_{ij}
\left(X_1^i X_2 ^j -X_3^iX_0^j\right)X_4^{d-i-j}\right]$$
and the surface $\widetilde{\mathcal{S}} : F(X_4,X_1,X_2,X_3,X_0)=0$ is projectively equivalent to $\mathcal{S}_f$.

We study now the tangent cone $\Omega$ of $\widetilde{\mathcal{S}}$ at the origin, whose equation depends on $\overline{j}$.

We distinguish three cases.
\begin{enumerate}
\item[(i)]\label{Case_i} $\overline{j}\geq 3$. The tangent cone $\Omega$ reads $(X_1-X_3)\sum_{i}a_{i,\overline{j}}X_3^iX_4^{d-i-\overline{j}}$.
\item[(ii)] $\overline{j}= 2$. Then $\Omega$ reads $X_4^{d-1}-(X_1-X_3)\sum_{i}a_{i,2}X_3^iX_4^{d-i-2}$.
\item[(iii)] $\overline{j}\leq 1$. The tangent cone reads $X_4^{d-1}$.
\end{enumerate}

If case (i) holds, then $X_1-X_3=0$ is a non-repeated $\mathbb{F}_q$-rational component of such a tangent cone. Also, there exists a non-singular point $P$ for $\Omega$.   By Lemma \ref{Lemma:DanYue2020} there exists an $\mathbb{F}_q$-rational plane $\pi$ through the line $\ell$ joining the origin  and $P$ such that $\pi\cap \widetilde{\mathcal{S}}$ has $\ell$ as a non-repeated tangent $\mathbb{F}_q$-rational line at the origin and $\pi\cap \widetilde{\mathcal{S}}$ has a non-repeated absolutely irreducible $\mathbb{F}_q$-rational component. Using twice Lemma \ref{Lemma:Aubry}, one deduces the existence of an absolutely irreducible $\mathbb{F}_q$-rational component in $\widetilde{\mathcal{S}}$ and therefore in $\mathcal{S}_f$.

If case (ii) holds, then $\Omega$ is a rational surface ($X_1$ has degree $1$). This means that $\Omega$ contains an absolutely irreducible $\mathbb{F}_q$-rational component of degree $1$ in $X_1$. It is readily seen that affine  singular points of $\Omega$ satisfy $X_4=0$ and therefore there exists at least a   non-singular point $P$ for $\Omega$. The same conclusion as in case (i) holds. 
\end{proof}

\begin{proposition}\label{Prop:Final_p_even}
Let $\mathcal{S}_f: F(X_0,X_1,X_2,X_3,X_4)=0$, with $F$  as in \eqref{Eq:S_f}, and $p=2$.  If  there exists an  $a_{i,j}\neq 0$, with $(i,j)\notin \{(1,0),(0,1)\}$, then  $\mathcal{S}_f$ contains an absolutely irreducible component defined over $\mathbb{F}_q$.
\end{proposition}
\proof

From  Proposition \ref{Prop:j_bar}, if $\overline{j}>1$ then $\mathcal{S}_f$ contains an absolutely irreducible component defined over $\mathbb{F}_q$. Therefore we can suppose that $\overline{j}\leq1$. Suppose that $d>1$.

Consider again Equation \eqref{Eq:F} and let $\mathcal{W}=\mathcal{S}_f \cap (X_2=X_4)$ be defined by 
\begin{eqnarray*}
&& (X_1+X_3) \left[\sum_{i,j} a_{ij}X_2^jX_0^{d-i-j}
\left(X_1^i  +X_3^i\right)\right]\\
&=&(X_1+X_3) \left( X_2 \sum_{0< i\leq d-2} a_{i,1}(X_1^i  +X_3^i)X_0^{d-i-1}+ \sum_{0< i\leq d} a_{i,0}(X_1^i  +X_3^i)X_0^{d-i}\right).
\end{eqnarray*}
If some of the $a_{i,1}$'s, $0<i\leq d-2$, is nonzero then  $\mathcal{W}$ contains a rational component of degree $1$ in $X_2$ and therefore $\mathcal{S}_f$ contains an absolutely irreducible component defined over $\mathbb{F}_q$ through it. 

So, $a_{i,1}=0$ for $i>0$ and  $F(X_0,X_1,X_2,X_3,X_4)$ reads

$$(X_2 +X_4)^2X_0^{d-1} + a_{0,1}(X_1+X_3)(X_2+X_4)X_0^{d-1} +(X_1+X_3)\sum_{0<i\leq d}  a_{i,0}
\left(X_1^i  +X_3^i\right)X_0^{d-i}.$$
If $a_{0,1}=0$ then $\mathcal{S}_f$ is clearly a rational surface (in $Z=(X_2 +X_4)^2$) and therefore it  contains an absolutely irreducible component defined over $\mathbb{F}_q$.

Suppose now $a_{0,1}\neq 0$ and consider now the affine part of $\mathcal{S}_f: F(1,X_1,X_2,X_3,X_4)=0$. It can be written as 
\begin{equation}\label{Eq:Y}
Z^2+Z+\frac{\sum_{0<i\leq d}  a_{i,0}
\left(X_1^i  +X_3^i\right)}{a_{0,1}^2(X_1+X_3)},
\end{equation}
where $Z=(X_2 +X_4)/(a_{0,1}(X_1+X_3)).$ 

In what follows we will prove that the surface $\mathcal{Y}$ defined by the affine equation \eqref{Eq:Y} is absolutely irreducible. 

Note that  the surface $Z^2+Z+\frac{\sum_{0<i\leq d}  a_{i,0}
\left(X_1^i  +X_3^i\right)}{a_{0,1}^2(X_1+X_3)}=0$ is reducible only if there exists  $\overline{Z}(X_1,X_3)=\frac{\overline{Z}_1(X_1,X_3)}{\overline{Z}_2(X_1,X_3)}$, with $\overline{Z}_i(X_1,X_3)\in \overline{\mathbb{F}}_q[X_1,X_3]$, such that $(\overline{Z}(X_1,X_3))^2+\overline{Z}(X_1,X_3)+\frac{\sum_{0<i\leq d}  a_{i,0}
\left(X_1^i  +X_3^i\right)}{a_{0,1}^2(X_1+X_3)}=0$. This condition reads 
\begin{equation}\label{Eq:Z}
(\overline{Z}_1(X_1,X_3))^2+\overline{Z}_1(X_1,X_3)\overline{Z}_2(X_1,X_3)+(\overline{Z}_2(X_1,X_3))^2\frac{\sum_{0<i\leq d}  a_{i,0}
\left(X_1^i  +X_3^i\right)}{a_{0,1}^2(X_1+X_3)}=0.
\end{equation}
Now, $\deg(\overline{Z}_1)>\deg(\overline{Z}_2)$, otherwise  $2\deg(\overline{Z}_1)\leq \deg(\overline{Z}_1)+\deg(\overline{Z}_2)<2\deg(\overline{Z}_2)+d-1$ and \eqref{Eq:Z} cannot be satisfied. 

Also, the same argument shows that $2\deg(\overline{Z}_1)=2\deg(\overline{Z}_2)+d-1$ and therefore the highest homogeneous parts in $(\overline{Z}_1(X_1,X_3))^2$ and in $(\overline{Z}_2(X_1,X_3))^2\frac{\sum_{0<i\leq d}  a_{i,0}
\left(X_1^i  +X_3^i\right)}{a_{0,1}^2(X_1+X_3)}$ must be equal.

Write  $d=2^jk$, with $j\geq0$ and $k$ odd. The above argument yields that, in particular,
$$L_d(X_1,X_3)=\frac{\left(X_1^d  +X_3^d\right)}{(X_1+X_3)}=\frac{\left(X_1^k  +X_3^k\right)^{2^j}}{(X_1+X_3)},$$
the highest homogenous term in $\frac{\sum_{0<i\leq d}  a_{i,0}
\left(X_1^i  +X_3^i\right)}{a_{0,1}^2(X_1+X_3)}$,
is a square in  $\overline{\mathbb{F}}_q[X_1,X_3]$. 
\begin{itemize}
    \item If $j>0$ then $(X_1+X_3)^{2^j-1}$ is the highest power of $(X_1+X_3)$ dividing $L_d(X_1,X_3)$. Since  $2^j-1$ is odd,    $L_d(X_1,X_3)$ is not a square in $\overline{\mathbb{F}}_q[X_1,X_3]$ and the surface $\mathcal{Y}$ is irreducible. 
\item If  $j=0$, following the same argument as above, if $\mathcal{Y}$ is reducible then $k=1$, a contradiction to $d>1$.
\end{itemize}
This shows that if $d>1$ then $\mathcal{Y}$ and therefore $\mathcal{S}_f$ is absolutely irreducible or contains an absolutely irreducible component defined over $\mathbb{F}_q$. The claim follows.
\endproof

\begin{proposition}\label{Prop:Final_p_odd}
Let $\mathcal{S}_f: F(X_0,X_1,X_2,X_3,X_4)=0$, with $F$  as in \eqref{Eq:S_f},  $p>2$, and $a_{0,1}=0$.
If $d\neq p^j$, with $j\geq 0$, or $a_{i,j}=0$ for some $(i,j)\neq (d,0)$,  then $\mathcal{S}_f$ contains an absolutely irreducible component defined over $\mathbb{F}_q$.
\end{proposition}
\proof
From  Proposition \ref{Prop:j_bar}, if $\overline{j}>1$ then $\mathcal{S}_f$ contains an absolutely irreducible component defined over $\mathbb{F}_q$. Therefore we can suppose that $\overline{j}\leq1$.

Suppose that $\overline{j}=1$. Consider $\widetilde{\mathcal{S}} : F(X_4,X_1,X_2,X_3,1)=0$ defined by 
\begin{eqnarray*}
(X_2 -1)^2X_4^{d-1} + (X_1-X_3) \left[\sum_{i} a_{i,1}
\left(X_1^i X_2  -X_3^i\right)X_4^{d-i-1}\right]&&\\+(X_1-X_3) \left[\sum_{i} a_{i,0}
\left(X_1^i   -X_3^i\right)X_4^{d-i}\right]&=&0.
\end{eqnarray*}
Let $\mathcal{W}=\widetilde{\mathcal{S}} \cap (X_1=0)$ be defined by 
\begin{eqnarray*}
(X_2 -1)^2X_4^{d-1} +X_3 \left[\sum_{i>0} a_{i,1}
 X_3^iX_4^{d-i-1}\right]&&\\-a_{0,1}X_3 (X_2-1)X_4^{d-1} +X_3\left[\sum_{i} a_{i,0}
X_3^iX_4^{d-i}\right]&=&0.
\end{eqnarray*}

 The above surface is of degree two in $(X_2 -1)$ and therefore, since $p$ is odd,  it is reducible if and only if 

\begin{eqnarray*}
a_{0,1}^2X_3^2 X_4^{2d-2}-4X_4^{d-1}X_3\left[\sum_{i>0} a_{i,1}
 X_3^iX_4^{d-i-1}+\sum_{i} a_{i,0}
X_3^iX_4^{d-i}\right]&=&\\
X_3^2 X_4^{d-1}\left(a_{0,1}^2X_4^{d-1}-4\left[\sum_{i>0} a_{i,1}
 X_3^{i-1}X_4^{d-i-1}+\sum_{i} a_{i,0}
X_3^{i-1}X_4^{d-i}\right]\right)
\end{eqnarray*}
is  a square in $\overline{\mathbb{F}}_q[X_3,X_4]$. If $a_{i,1}\neq 0$ for some $i>0$ then  the degrees of the monomials in $G(X_3,X_4)=a_{0,1}^2X_4^{d-1}-4\left[\sum_{i>0} a_{i,1}
 X_3^{i-1}X_4^{d-i-1}+\sum_{i} a_{i,0}
X_3^{i-1}X_4^{d-i}\right]$ are $d-2$ and $d-1$ and both the values are assumed. In this case, $G(X_3,X_4)$ cannot be  a square in $\overline{\mathbb{F}}_q[X_3,X_4]$. Therefore $\mathcal{W}$ is absolutely irreducible and, by Lemma \ref{Lemma:Aubry},  $\widetilde{\mathcal{S}}$ (and therefore $\mathcal{S}$) contains an absolutely irreducible component defined over $\mathbb{F}_q$.

Suppose now that $a_{i,1}= 0$ for all $i>0$. So, $F(X_0,X_1,X_2,X_3,X_4)$ equals
$$(X_2 -X_4)^2X_0^{d-1} + (X_1-X_3) \left[\sum_{i} a_{i,0}
\left(X_1^i-X_3^i\right)X_0^{d-i}\right].$$

In this case, consider the affine part of $\mathcal{S}_f$ defined by $F(1,X_1,X_2,X_3,X_4)=0$. Note that, as already observed, $\mathcal{S}_f$ is reducible if and only if  
$$(X_1-X_3) \left[\sum_{i} a_{i,0}
\left(X_1^i-X_3^i\right)\right]$$ 
is a square in $\overline{\mathbb{F}}_q[X_3,X_4]$. 

In what follows we will show that if $L(X_3,X_4)=(X_1-X_3) \left[\sum_{i} a_{i,0}\left(X_1^i-X_3^i\right)\right]$  is  a square in $\overline{\mathbb{F}}_q[X_3,X_4]$ then there exists no $i$ with $0<i<d$ such that $a_{i,0}\neq 0$.

Suppose that $L(X_3,X_4)$ is a  square in $\overline{\mathbb{F}}_q[X_3,X_4]$. Then, the highest homogenous term $L_d(X_3,X_4)=a_{d,0}(X_1-X_3)(X_1^d-X_3^d)$  is  a square in $\overline{\mathbb{F}}_q[X_3,X_4]$ too. Write $d=kp^j$, with $p\nmid k$. Then $L_d(X_3,X_4)=a_{d,0}(X_1-X_3)\left(X_1^k-X_3^k\right)^{p^j}$. 
\begin{itemize}
    \item  If $k>1$ then there are other factors $(X_3-\eta X_4)$, $\eta^k=1$, $\eta\neq 1$, such that $(X_3-\eta X_4)^{p^j}$ is the highest power of $(X_3-\eta X_4)$ dividing $L_d(X_3,X_4)$, a contradiction to $L_d(X_3,X_4)$ being  a square. 
    \item Let now $k=1$ and $d=p^j$ for some $j\geq 0$. Suppose that there exists another $0<i<d$ such that $a_{i,0}\neq 0$ and let $\overline{i}=\max\{i : a_{i,0}\neq0, 0<i<d\}$. Write 
\begin{eqnarray*}
(X_1-X_3) \left[\sum_{i} a_{i,0} \left(X_1^i-X_3^i\right)\right]
&=&L_d(X_3,X_4)+L_{\overline{i}}(X_3,X_4)+\cdots\\
&=&\left((X_1-X_3)^{(p^j+1)/2}+M(X_3,X_4)+\cdots\right)^2,
\end{eqnarray*}
where $M(X_3,X_4)$ has degree $\overline{i}-(p^j+1)/2$. Now 
$$2(X_1-X_3)^{(p^j+1)/2}M(X_3,X_4)=L_{\overline{i}}(X_3,X_4)=a_{\overline{i},0}(X_1-X_3)(X_1^{\overline{i}}-X_3^{\overline{i}})$$
and then $(X_1-X_3)^{(p^j-1)/2}\mid (X_1^{\overline{i}}-X_3^{\overline{i}})$. Write $\overline{i}=k_1p^{j_1}$, so that $(X_1^{\overline{i}}-X_3^{\overline{i}})=(X_1^{k_1}-X_3^{k_1})^{p^{j_1}}$. Since  $(X_1-X_3)^{p^{j_1}}$ is the highest power of $(X_1-X_3)$ dividing  $(X_1^{k_1}-X_3^{k_1})^{p^{j_1}}$, we must have $(p^j-1)/2\leq p^{j_1}$. Since $j_1\leq j-1$, 
$$p^j-1\leq 2p^{j_1}\leq 2p^{j}/p,$$
and therefore $p^j\leq \frac{p}{p-2}$. This yields $p^j=3$, that is $j=1$ and $p=3$. 
Now, $(X_1-X_3)^4+\alpha (X_1-X_3)(X_1^2-X_3^2)+\beta(X_1-X_3)^2$, $\alpha,\beta \in \mathbb{F}_q$, is a square in $\overline{\mathbb{F}_{q}}[X_1,X_3]$ only if $(X_1-X_3)^2+\alpha (X_1+X_3)+\beta$ is a square. By direct checking, this is possible only if $\alpha=\beta=0$. 
\end{itemize}
This shows that if there exists an $i$, $1<i<d$, such that $a_{i,0}\neq 0$ 
then $L(X_3,X_4)$ is not a square in $\overline{\mathbb{F}}_q[X_3,X_4]$ and therefore $\mathcal{S}_f$ is absolutely irreducible. \endproof

We are now in position to prove the main result of this paper.
\begin{theorem}\label{Th:Main2}
Let $f(X,Y)=\sum_{ij}a_{i,j}X^iY^j$, with $a_{0,1}=0$ if $p>2$. Suppose that $q>6.3 (d+1)^{13/3}$, where $\deg(f)=d$. If $\mO_4(f)$ is an ovoid of $\Q$ then one of the following holds:
\begin{enumerate}
    \item $p=2$ and $f(X,Y)=a_{1,0}X+a_{0,1}Y$. Hence $O_4(f)$ is an elliptic quadric.
    \item $p>2$ and $f(X,Y)=a_{p^j,0}X^{p^j}$, $j\geq 0$.  Hence $O_4(f)$ is either and elliptic quadric or a Kantor ovoid.
\end{enumerate}
\end{theorem}
\proof
By combining Propositions \ref{Prop:Final_p_even} and \ref{Prop:Final_p_odd} with Theorem \ref{Th:Main} we obtain that either $p=2$ and  $f(X,Y)=a_{1,0}X+a_{0,1}Y$, or $p>2$ and $f(X,Y)=a_{p^j,0}X^{p^j}$, $j\geq 0$. The claim then follows by applying Condition \eqref{Eq:intro}. 
\endproof

Note that the ovoids corresponding to the elliptic quadric  
$$f_1(x,y)=ax+y, \qquad tr(a)= 1,$$
and the Tits ovoid
$$f_2(x,y)=x^{2^{h+1}+1}+ y^{2^{h+1}}, \qquad q=2^{2h+1}.$$
yield the following equation for  $\mathcal{S}_f$ 

\begin{equation}\label{Eq:F_1}
(X_2 +X_4)^2 + (X_1+X_3) (aX_1+X_2+aX_3+X_4)=0
\end{equation}
and 
\begin{equation}\label{Eq:F_2}
(X_2 +X_4)^2X_0^{2^{h+1}}+ (X_1+X_3) (X_1^{2^{h+1}+1}+X_2^{2^{h+1}}X_0+X_3^{2^{h+1}+1}+X_4^{2^{h+1}}X_0)=0.
\end{equation}
In particular,  \eqref{Eq:F_1}  defines a reducible quadric splitting into two linear factors defined over $\mathbb{F}_{q^2}$. It has been already observed that  Main Theorem  does not apply to $f_2(x,y)$, since $\deg(f_2)\simeq  \sqrt{q}$. In this case, it is possible to prove that $\mathcal{S}_f$ is absolutely irreducible.

\section{Two classes of permutation polynomials in characteristic three}\label{Sec:permPol}
As a by-product of our investigation on ovoids of $Q(4,q)$, we provide in this section two families of permutation polynomials in characteristic three. A polynomial $f(x)\in  \mathbb{F}_q[x]$ is a \emph{permutation polynomial} (PP) if it is a bijection of the finite field $\mathbb{F}_q$ into itself. Permutation polynomials were first studied by Hermite and Dickson; see \cite{MR1502214,hermite1863fonctions}.Particular, simple structures, or additional extraordinary properties are usually required by applications of PPs in other areas of mathematics and engineering, such as cryptography, coding theory, or combinatorial designs. In this case, permutation polynomials meeting these criteria are usually difficult to find. For a deeper introduction to the connections of PPs with other fields of mathematics we refer to \cite{MR3293406,MR3087321,wang201915,MR2494389} and the references therein.

In this section we will assume that $f(x,y) = \sum_i a_i x^{p^i} + \sum_j b_j y^{p^i}$ is the sum of two $\fp$-linear functions the first  one only in $x$ and the other one only in $y$.
Note that all known examples of ovoids $\mO_4(f)$,  apart from the Tits and Ree-Tits slice, arise from functions $f(x,y)$ of this type.
We may assume that there exist $\overline{i}$ such that $a_{\overline{i}} \ne 0$ and $\overline{j}$ such that
$b_{\overline{j}} \ne 0$.
By using the same notation as in previous section we have that in this case the hypersurface  $\mathcal{S}_f$
of  $ \mathrm{PG}(4,q)$ has equation  $F(X_0,X_1,X_2,X_3,X_4)=0$, where $F(X_0,X_1,X_2,X_3,X_4)$ equals
$$
X_0^{d-2}(X_2 -X_4)^2 + (X_1-X_3) \left[\sum_{i} a_{i}
 X_0^{d-1-p^i}(X_1-X_3)^{p^i} + \sum_{j} b_{j}X_0^{d-1-p^j}(X_2-X_4)^{p^j}  \right].
$$
Put $X_0=1$, $L:=X_2-X_4$ and  $M:=X_1-X_3$.  {Then $F(X_0,X_1,X_2,X_3,X_4)=0$ reads}  
\begin{equation}{\label{LM}}
 L^2+\sum_i a_i M^{p^i+1} + M\sum_j b_j L^{p^j} =0.
\end{equation}
\\
Every examples with  $f(x,y) = \sum_i a_i x^{p^i} + \sum_j b_j y^{p^i}$   gives an equation  in $L$ and $M$ of this type
that has solution if and only if $L=M=0$.
\\
Fixing a basis $\{\xi,\xi^p,\xi^{p^2},\ldots,\xi^{p^{h-1}} \}$ of $\fq$ over $\fp$, every element $a\in\fq$ can be written as $a=a_0\xi+a_1 \xi^p +a_2 \xi^{p^2}+\ldots + a_{h-1}\xi^{p^{h-1}}$ for some $a_i\in \fp$.
Moreover, by putting $L_i:=L^{p^i}$ and $M_i:=M^{p^i}$,  Equation \eqref{LM} reads
\[ L_0^2 + M_0\left(\sum_i a_i M_i + \sum_j b_j L_j\right) =0. \] 
In the projective space $\PG(2h-1,q)$ with homogeneous coordinates \\ $(L_0,L_1,\ldots, L_{h}, M_0,M_1,\ldots, M_{h})$
the previous equation gives a quadric that  is always a cone with vertex
a $(2h-1)$-dimensional subspace $V$ given by:
\begin{equation}\label{V} V: L_0=M_0= \sum_i a_i M_i + \sum_j b_j L_j =0\end{equation}
and projecting a non-degenerate conic  $\Gamma : L_0^2+a_{\overline{i}} M_0M_{\overline{i}} = 0$ in a plane $\pi$ given by:
\begin{equation}\label{gamma} \pi: L_j=0, j=1,\ldots,h-1, M_i=0, i=1, \ldots h-1 \mbox{ with } i\ne \overline{i}.
\end{equation}
Every example of ovoid of $Q(4,q)$ gives a  corresponding cone that is skew with the subgeometry  $\Sigma=\PG(2h-1,p)$ of
$\PG(2h-1,q)$ given by
\begin{equation}\label{Sigma} \Sigma = \{(a,a^p,\ldots,a^{p^{h-1}},b,b^p,\ldots,b^{p^{h-1}})\}_{a,b \in \fq} \mbox{ with  } (a,b) \ne (0,0).\end{equation}
The condition that the cone is disjoint from the subgeometry $\Sigma$ reads as follows:
The following equation in $x$ should have no solution unless $x=0$ for every $(l_0,m_0)\ne (0,0)$
\[  \sum_{j} b_j m_0 \ell_0^{p^j} x^{p^j} +
\sum_{i} a_i m_0^{p^i+1}x^{p^i} + \ell_0^2 x  =0.\]

From the previous arguments we obtain the following

\begin{proposition}
There exists an ovoid of $Q(4,q)$, $q=p^{h}$, as in \eqref{Eq:Param} with $f(x,y)= \sum_i a_i x^{p^i} + \sum_j b_j y^{p^i}$ if and only if in $\PG(2h-1,q)$ there exists a quadratic cone with vertex a $(2h-4)$-dimensional subspace $V$ as in \eqref{V} and base a non-degenerate conic $\Gamma$ as in \eqref{gamma} skew with the subgeometry $\Sigma\cong \PG(2h-1,p)$ as in \eqref{Sigma}.
\end{proposition}

 Finally, as a byproduct of the above proposition, we obtain two families of $\ft$-linearized  permutation polynomials from the  Pentilla-William ovoid of $Q(4,3^5)$ and  from the Thas-Payne ovoids of $Q(4,3^h)$, $h>2.$ 

\begin{proposition}
Let $q=3^5$, for every $(\ell_0,m_0)\in \fq^2\setminus \{(0,0)\}$ the following are $\ft$-linearized permutation polynomials:
\[ m_0\ell_0^{81} x^{81}+m_0^{10} x^{9}   -\ell_0^2 x =0.\]
\end{proposition}


\begin{proposition}
Let $q=3^n$, $n>2$ and let $m$ be a fixed non-square in $\fq$.  For every $(\ell_0,m_0)\in \fq^2\setminus \{(0,0)\}$ the following are $\ft$-linearized permutation polynomials:
\[(m m_0^2 - \ell_0^2)^9 x^9  + m_0^9\ell_0^{3} x^{3}+ (m^{-1})m_0^{10} x. \]
\end{proposition}
\noindent 

\section*{Acknowledgments}
The research of both Daniele Bartoli and Nicola Durante was supported by the Italian National Group for Algebraic and Geometric Structures and their Applications (GNSAGA - INdAM). We thank the anonymous reviewers for their careful reading of our manuscript and their  comments and suggestions.

\end{document}